\documentclass[a4paper,12pt,oneside]{article}

\usepackage{eucal}
\usepackage{color}
\usepackage[all]{xy}
\usepackage[centertags]{amsmath}
\usepackage{latexsym}
\usepackage{stmaryrd}
\usepackage{amsfonts}
\usepackage{amssymb}
\usepackage{amsthm}
\usepackage{fancyhdr}
\usepackage{t1enc}
\usepackage[utf8]{inputenc} 
\usepackage[english]{minitoc}
\usepackage{fancyhdr}
\usepackage{newlfont}
\usepackage[active]{srcltx}
\usepackage{graphicx}
\usepackage{t1enc}
\usepackage{tikz}
\usepackage{makeidx}
\usepackage{appendix}
\usepackage[english]{babel}			
\usepackage{float}
\usepackage{fancyhdr}
\newcommand\shorttitle{Combinatorial K-systoles on a punctured torus and a pair of pants}
\newcommand\authors{E. A. A. Diop, M. Gaye, A. K. Sane}
\theoremstyle{plain}
\newtheorem{theorem}{Theorem}
\newtheorem*{conj}{Conjecture}
\newtheorem*{thm}{Theorem}
\newtheorem{r emark}[theorem]{Remark}
\newtheorem{lemma}{Lemma}
\newtheorem{exercise}[theorem]{Exercise}
\newtheorem{proposition}{Proposition}
\newtheorem{definition}[theorem]{Definition}
\newtheorem{corollary}[theorem]{Corollary}

\newtheorem{example}[theorem]{Exemple}
\newtheorem{prop}{Propri\'et\'e}[section]

\newcommand{\Z}{\mathbb Z}

\newcommand{\N}{\mathbb N}

\newcommand{\D}{\mathbb D}

\newcommand{\G}{\mathcal{G}}

\newcommand{\ds}{\displaystyle}

\newcommand{\beq}{\begin{equation}}
\newcommand{\eeq}{\end{equation}}
\newcommand{\beqn}{\begin{eqnarray}}
\newcommand{\eeqn}{\end{eqnarray}}
\newcommand{\bpro}{\begin{proposition}}
\newcommand{\epro}{\end{proposition}}
\newcommand{\blem}{\begin{lemma}}
\newcommand{\elem}{\end{lemma}}
\newcommand{\bdfn}{\begin{definition}}
\newcommand{\edfn}{\end{definition}}
\newcommand{\bcor}{\begin{corollary}}
\newcommand{\ecor}{\end{corollary}}
\newcommand{\bthm}{\begin{theorem}}
\newcommand{\ethm}{\end{theorem}}
\newcommand{\bex}{\begin{example}}
\newcommand{\eex}{\end{example}}
\newcommand{\brmq}{\begin{remark}}
\newcommand{\ermq}{\end{remark}}
\newcommand{\benum}{\begin{enumerate}}
\newcommand{\eenum}{\end{enumerate}}
\newcommand{\bitem}{\begin{itemize}}
\newcommand{\eitem}{\end{itemize}}
\newcommand{\bexer}{\begin{exercise}}
\newcommand{\eexer}{\end{exercise}}
\newcommand{\bproof}{\begin{proof}}
\newcommand{\eproof}{\end{proof}}
\newcommand{\eprop} {\end{prop} }
\newcommand{\bprop}{\begin{prop}}
\theoremstyle{plain}

\usetikzlibrary{arrows}
\usetikzlibrary{decorations.markings}
\tikzstyle directed=[postaction={decorate,decoration={markings,
    mark=at position .65 with {\arrow{latex}}}}]

\providecommand{\thanks}[1]

\usepackage[english]{babel}
\title{Combinatorial $k$-systoles on a punctured torus and a pair of pants}
\author{ElHadji Abdou Aziz Diop, Masseye Gaye,  Abdoul Karim Sane}

\begin{document}
\maketitle						
\abstract
In this paper, $S$ denotes a surface homeomorphic to a punctured torus or a pair of pants. Our interest is the study of \emph{\textbf{combinatorial $k$-systoles}} that is closed curves with self-intersection numbers greater than $k$ and with least combinatorial length. We show that the maximal intersection number $I^c_k$ of combinatorial $k$-systoles of $S$ grows like ~$k$ and $\underset{k\rightarrow+\infty}{\limsup}(I^c_k-k)=+\infty$.

This result, in case of a pair of pants and a punctured torus, is a positive response to the combinatorial version of the Erlandsson - Parlier conjecture, originally formulated for the geometric length.

\section{Introduction}

Let $\Sigma$ be a hyperbolic surface of finite type and $\mathcal{G}$ the set of primitive closed geodesics on $\Sigma$. For $\gamma\in\mathcal{G}$, let $l(\gamma)$ denotes its geometric length namely its length relatively to the hyperbolic metric and $i(\gamma,\gamma)$ its self-intersection number. For every $k\in\N$, we consider the set $$\mathcal{G}_{\geq k}=\{\gamma\in\mathcal{G}:\ i(\gamma,\gamma)\geq k \},$$ of closed geodesics on $\Sigma$ that self-intersects at least $k$ times and we set $$s_k:=\inf\{l(\gamma), \gamma\in\G_{\geq k}\}.$$ By discreteness of the length spectrum, the infimum is achieved by some geodesics. A \emph{\textbf{$k$-systole}} is a closed geodesic in $\G_{\geq k}$ denoted by $\gamma_k$ such that $l(\gamma_k)=s_k$. The $k$-systoles may have different self-intersection numbers and understanding the self-intersection spectrum of $k$-systoles is an interesting question. For instance when $k=0$, by suitably smoothing self-intersection points, one can make a closed geodesic simple and shorter. It follows from this heuristic argument that systoles (0-systoles) are simple (unless on a pair of pants with three cusps where $0$-systoles have self-intersection number equal to one). For $k=1$, Buser \cite{Bus} showed that $1$-systoles have one self-intersection point.  

When $k\geq 2$, the self-intersection spectrum of $k$-systoles is more complicated.

In \cite{erlandsson2016short}, V. Erlandson and H. Parlier studied:
$$I_k(\Sigma):=\max\{i(\gamma_k,\gamma_k)\};$$
the maximal self-intersection number of $k$-systoles introduced by Basmajian and Buser. They obtained the following:  
\begin{thm}[Erlandsson-Parlier \cite{erlandsson2016short}]
Let $\Sigma$ be an orientable complete hyperbolic surface with non-abelian fundamental group. Then
$$I_k(\Sigma)\leq 31\displaystyle{\sqrt{k+\frac{1}{4}}\left(16\sqrt{k+\frac{1}{4}}+1\right)}$$
\end{thm}
Previous works due to A. Basmajian \cite{Bas1} provide an upper bound depending on the topology of the surface when the surface is complete, finite area and finite type. 

Erlandsson and Parlier also showed that when $\Sigma$ is a complete orientable hyperbolic surface of finite type with at least one cusp, there exists two constants $D(\Sigma)$ and $K(\Sigma)$ depending on the geometry (the metric) on $\Sigma$ such that: $$I_k(\Sigma)\leq k+D(\Sigma)\log(k)$$ for all $k>K(\Sigma).$

 T. H. Vo \cite{vo2019short} proved the best known result (at the time of writing) by showing that when $\Sigma$ has at least one cusp  there is a constant $K(\Sigma)$ such that for $k\geq K(\Sigma)$, $k$-systoles are $\emph{\textbf{exact}}$ namely $I_k=k$.

These results imply that for surfaces with at least one cusp, we have: $$\lim\limits_{k\to+\infty}\ds\frac{I_k(\Sigma)}{k}=1\quad\hspace{0.1cm} \mathrm{and}\hspace{0.1cm}\lim\limits_{k\to+\infty} (I_k(\Sigma)-k)=0.$$

\begin{conj}[Erlandsson-Parlier \cite{erlandsson2016short}]
If $\Sigma$ is a compact (and in particular has no cusp) hyperbolic surface, then 
$$\underset{k\rightarrow+\infty}{\limsup} \hspace{0.2cm}(I_k(\Sigma)-k)=+\infty.$$
\end{conj}

In this article, we consider a combinatorial version of Erlandsson-Parlier conjecture and only for a punctured torus and a pair of pants. 

If $S$ is a surface homeomorphic to a pair of pants or a punctured torus then, $\pi_1(S)$ is a free group of rank 2. Let $\gamma$ be a closed curve in $S$. We refer to the \textit{\textbf{combinatorial length}} of $\gamma$, and we denote it by $L(\gamma)$, as the length of the reduced word $w_{\gamma}$ that corresponds to the homotopy class of $\gamma$ in $\pi_1(S)$ with respect to the standard generating set. The self-intersection number of $\gamma$ is the minimal number of time $\gamma$ self-intersects in its free homotopy class. A  \textbf{\textit{combinatorial $k$-systole}} is a closed curve with self-intersection number greater than $k$ and with the smallest combinatorial length among all closed curves with self-intersection greater than $k$. We denote by $I^c_k$ the maximal self-intersection number of combinatorial $k$-systoles.

\noindent The problem of geometric $k$-systoles translates into a combinatorial one. Turning geometric problems to combinatorial ones is a common practice and one can found examples of this in the works of M. Chas and al.\cite{CCP}, \cite{CP}. Our first result answers to the combinatorial version of Erlandsson-Parlier conjecture:
\begin{theorem}\label{thm1}
Let $S$ be a surface homeomorphic to a pair of pants or a punctured torus. Then,

$$\underset{k\rightarrow+\infty}{\liminf} \hspace{0.2cm}I^c_k-k=0\ \ \mbox{and} \ \ \underset{k\rightarrow+\infty}{\limsup} \hspace{0.2cm}I^c_k-k=+\infty.$$
\end{theorem}

One way to compute self-intersection numbers of closed curves on surfaces is to use Bowen-Series coding. First, we endow $S$ with a hyperbolic metric and in this case, each homotopy class of an essential closed geodesic contains a unique geodesic and the minimal self-intersection number of a closed curve in this homotopy class is achieved by this geodesic. If $\gamma$ is a closed curve in $S$ and $\alpha$ the unique geodesic in its homotopy class, The coding $w_{\gamma}$ of $\gamma$ is read off from a lift $\widetilde{\gamma}$ of $\gamma$ to the universal cover. From $w_{\gamma}$, one can derive an upper bound of $i(\gamma,\gamma)$ using an argument which we explain in  section \ref{sec2}.\\
Since our result does not depend on the metric on $S$, it remains true when $S$ has a cusp, which contrasts with T. H. Vo's result \cite{vo2019short} in the case of geometric $k$-systoles. This is not so surprising since geometric and combinatorial length are not equivalent when $S$ has at least one cusp.

For a hyperbolic surface without cusps, geometric length and combinatorial length of closed geodesic are coarsely equivalent \cite{Har}. We had expected that Theorem \ref{thm1} extends to the geometric case but it requires further work since the gap (between geometric and combinatorial length) is quite big and one loses a lot of 
information about  self-intersection. Nonetheless, we like to think our result as a first step toward Erlandsson-Parlier conjecture. We also show the following:
\begin{theorem}\label{thm2}
Let $S$ be a surface homeomorphic to a pair of pants or a punctured torus. Then,

$$\underset{k\rightarrow+\infty}{\lim}\frac{I^c_k}{k}=1.$$\end{theorem}

\begin{paragraph}{Outline of the paper:} Section ~\ref{sec2} is about Bowen-Series coding and self-intersection numbers of  closed geodesics when we endow $S$ with a hyperbolic metric. In Section ~\ref{sec3}, we give the proof of Theorem \ref{thm1} and Theorem \ref{thm2}.  
\end{paragraph}

\section{Preliminaries}\label{sec2}
In this section, we explain Bowen-Series coding of closed geodesics on $S$, and we explain how to derive their self-intersection numbers using it. 
\begin{paragraph}{Bowen-series coding of geodesics on $S$:}
Let ~$\D$ denote the Poincar\'e disk endowed with the hyperbolic metric and $O$ its center. For every hyperbolic isometry $h$, we denote by $D(h)$ the half-space bounded by the perpendicular bisector of the segment $[O,h(O)]$ and containing $h(O)$. Let $a$ and $b$ be two hyperbolic isometries of $\D$ with inverse $\bar{a}$ and $\bar{b}$ such that: $$[D(a)\cup D(\bar{a})]\cap [D(b)\cup D(\bar{b})]=\emptyset.$$

The group $\Gamma:=\langle a, b\rangle$ is a Schottky group of rank $2$ (see \cite{Bea}, \cite{Dal}) and $S:=\D/\Gamma$ is a hyperbolic surface homeomorphic to a punctured torus or a pair of pants, depending on the configuration of the domains $D(e)$; $e\in\{a, \bar{a}, b, \bar{b}\}$ (see Figure~\ref{sch}).

In this case, the group $\Gamma$ is a purely hyperbolic  Fuchsian group and the set $\mathcal{P}=\ds\bigcap_{e\in\{a, \bar{a}, b, \bar{b}\}}\D-\overset{\circ}{D}(e)$  is a fundamental domain of $\D$ for the action of $\Gamma$ on $\mathbb{D}$. 
\begin{figure}[htbp]
\begin{center}
\includegraphics[scale=0.35]{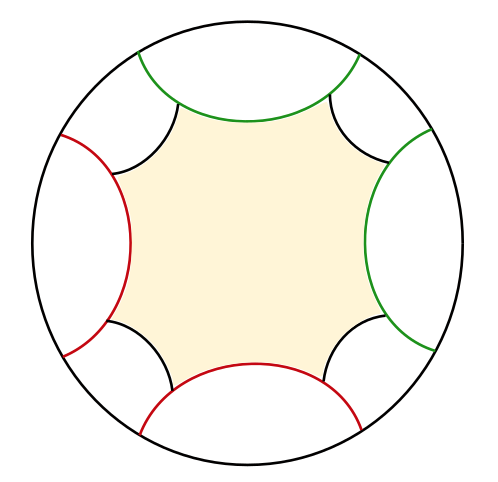}\hspace{2cm}
\includegraphics[scale=0.35]{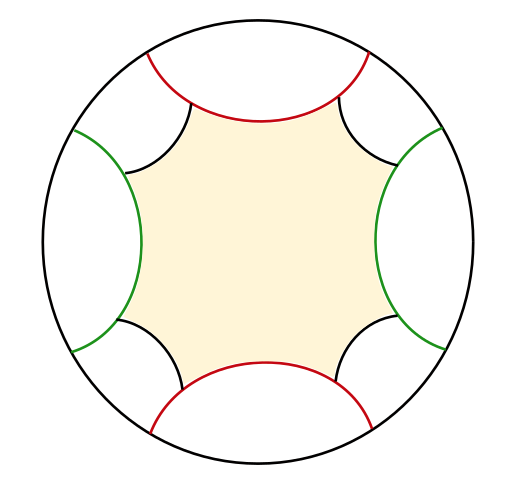}\vspace{0,3cm}\put(-210,-15){\Large{$\downarrow$}}\put(-48,-15){\Large{$\downarrow$}}\put(-174,40){\tiny{$D(a)$}} \put(-200,72){\tiny{$D(\bar{a})$}}\put(-233,40){\tiny{$D(b)$}}\put(-200,10){\tiny{$D(\bar{b})$}} \put(-24,40){\tiny{$D(a)$}} \put(-50,72){\tiny{$D(b)$}}\put(-83,40){\tiny{$D(\bar{a})$}}\put(-50,10){\tiny{$D(\bar{b})$}}\\
\includegraphics[scale=0.35]{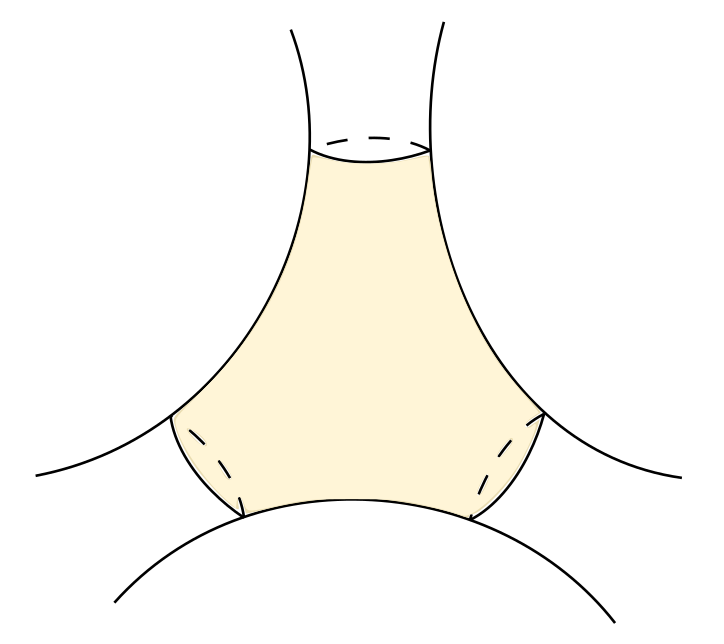}\hspace{2cm}
\includegraphics[scale=0.35]{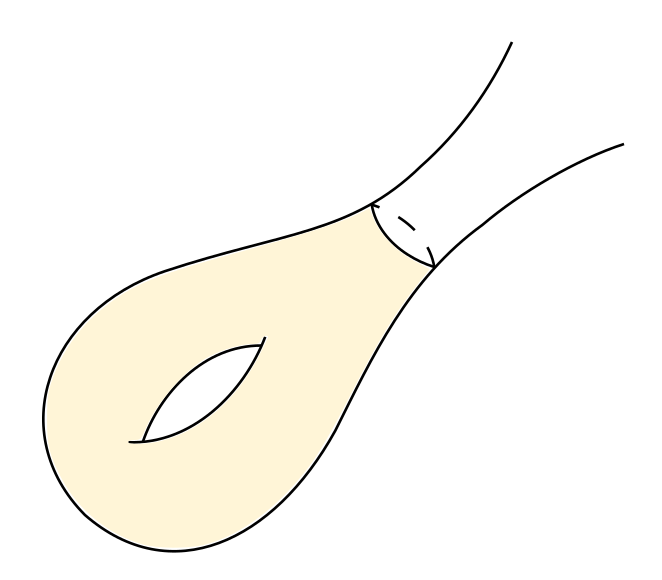}
\caption{\footnotesize{Configuration of the domains $D(a)$, $D(\bar{a})$, $D(b)$, $D(\bar{b})$ and the surfaces obtained in the quotient. The shaded regions are the smallest regions that contain all closed geodesic.}}
\label{sch}
\end{center}
\end{figure}
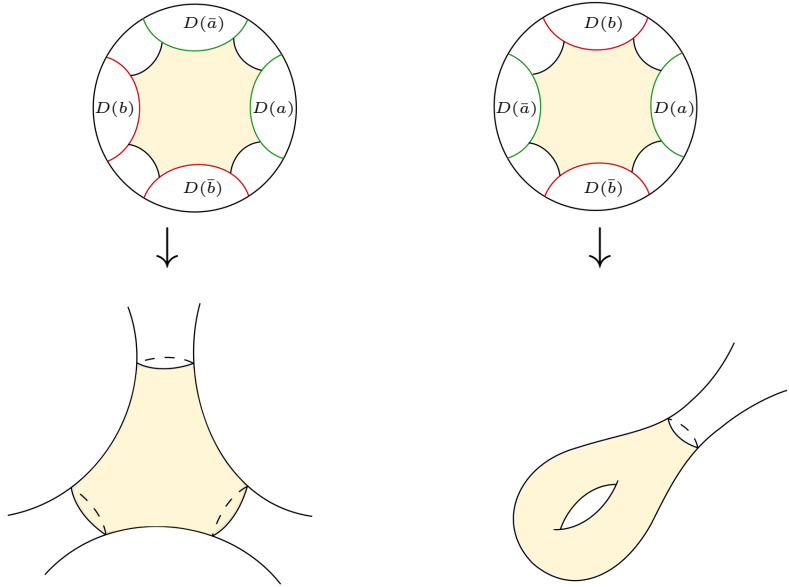

We denote by $W^+$ (respectively $W$) the set of infinite (respectively bi-infinite) reduced words on the alphabet $\{a, \bar{a}, b, \bar{b}\}$; and by $\sigma$ the shift map on $W$ defined by:
$$\sigma((e_i)_{i\in\Z}):=(e_{i+1})_{i\in\Z}.$$

Let $\Lambda:=\overline{\Gamma.O}\cap \partial \D$ be the limit set of $\Gamma$; where $\Gamma.O$ denotes the orbit of the origin $O$. 

The map 
\begin{align*}
\psi:W^+&\longrightarrow \Lambda\\
                     w:=e_0e_1e_2...&\longmapsto \underset{n\rightarrow+\infty}{\lim} e_0e_1e_2...e_n(0).
\end{align*}
is a bijection (see \cite{BS}) and for $\eta\in\Lambda$, we refer to $\psi^{-1}(\eta)$ as the coding of $\eta$.

Let $\widetilde{\gamma}$ be an oriented geodesic of $\mathbb{D}$ with endpoints $(\widetilde{\gamma}^-,\widetilde{\gamma}^+)\in\Lambda\times\Lambda$ and assume that $\widetilde{\gamma}$ intersects $\mathcal{P}$. Let $w_{\widetilde{\gamma}^+}=e_0e_1e_2...$ and $w_{\widetilde{\gamma}^-}=e_{-1}e_{-2}...$ be the coding of $\widetilde{\gamma}^+$ and $\widetilde{\gamma}^-$. Since $\gamma\cap\mathcal{P}\neq\emptyset$, then $e_0\neq e_1$ and the bi-infinite reduced word $W_{\widetilde{\gamma}}:=...\bar{e}_{-2}\bar{e}_{-1}e_0e_1e_2...$ is the coding of the geodesic $\widetilde{\gamma}$ when we fix $\mathcal{P}$ as the origin tile.

The word $W_{\widetilde{\gamma}}$ encodes the way $\widetilde{\gamma}$ crosses infinitely many copies of the tiles of $\mathbb{D}$. In fact, $\mathcal{P}$ tiles $\D$ and a lift $\widetilde{\gamma}$ that intersects $\mathcal{P}$ goes infinitely many times  across the tiles (following the two directions of $\widetilde{\gamma}$ and starting at $\mathcal{P}$). Moreover, if $g.\mathcal{P}$ and $h.\mathcal{P}$ are two adjacent tiles in $\D$, then $h=g.e$ with $e\in\{a,\bar{a}, b, \bar{b}\}$. Therefore, if $(g_n.\mathcal{P})_{n\in\Z}$ is the bi-infinite sequence of adjacent tiles crossed by $\widetilde{\gamma}$ with $g_0=\mathrm{Id}$, the coding associated to $\widetilde{\gamma}$ corresponds to: $$W_{\widetilde{\gamma}}:=(e_i)_{i\in\mathbb{Z}}:=...e_{-1}e_0e_1...$$
where $e_i\in\{a, \bar{a}, b, \bar{b}\}$ is such that $g_{i+1}=g_i.e_i$. 
The word $W_{\widetilde{\gamma}}$ is reduced, i.e, two consecutive letters do not cancel each other.

\begin{lemma}\label{lem}
Let $\widetilde{\gamma}$ and $\widetilde{\gamma}'$ be two geodesics of $\mathbb{D}$ whose endpoints are in $\Lambda$. Assume that $\widetilde{\gamma}$ and $\widetilde{\gamma}'$ intersect $\mathcal{P}$ and there is $g\in\Gamma$ such that $\widetilde{\gamma}'=g(\widetilde{\gamma})$. Then, then there exists $p\in\mathbb{Z}$ such that $W_{\widetilde{\gamma}'}=\sigma^p(W_{\widetilde{\gamma}})$.
\end{lemma}
\begin{proof} 
Let $W_{\widetilde{\gamma}}:=...e_{-1}e_0e_1...$ be the coding of $\widetilde{\gamma}$. So, the sequence of tiles crossing by $\widetilde{\gamma}$ is given by $(g_n(\mathcal{P}))_{n\in\mathbb{Z}}$ with $g_0=\mathrm{Id}$. Since  there is $g\in\Gamma$ such that $\widetilde{\gamma}'=g(\widetilde{\gamma})$,  $\widetilde{\gamma}'$ crosses the sequence $(h_n(\mathcal{P}))_{n\in\mathbb{Z}}$ where $h_n=g.g_n$ and $h_{n+1}=h_n.e_n$ for all $n\in \mathbb{Z}$. Since $\widetilde{\gamma}'$ crosses $\mathcal{P}$, there exists $p\in\mathbb{Z}$ such that $h_p=\mathrm{Id}$. Then, $W_{\widetilde{\gamma}'}:=...e'_{-1}e'_0e'_1...$ where $e'_{k}=e_{k+p}$ which implies that $W_{\widetilde{\gamma}'}=\sigma^p(W_{\widetilde{\gamma}})$.   
\end{proof}
Let $\gamma$ be a closed geodesic on $S$ and $\widetilde{\gamma}$ a lift of $\gamma$ that crosses $\mathcal{P}$. The endpoints of $\widetilde{\gamma}$ are in $\Lambda$ and   $W_{\widetilde{\gamma}}$ is periodic namely $W_{\widetilde{\gamma}}=...w_{\widetilde{\gamma}}w_{\widetilde{\gamma}}w_{\widetilde{\gamma}}...$ where $w_{\widetilde{\gamma}}:=e_0e_1...e_m\in\Gamma=\pi_1(S)$ is cyclically reduced, i.e, all the cyclic permutation of $w_{\widetilde{\gamma}}$ are reduced. The word $w_{\widetilde{\gamma}}$ also correspond to the homotopy class of  $\gamma$ in $S$ and the length of $w_{\widetilde{\gamma}}$  does not depend on the choice of the lift crossing $\mathcal{P}$ and we refer to it as  \textit{\textbf{the combinatorial length}} of $\gamma$. The Bowen-Series coding of a closed geodesic $\gamma$ is the cyclic class of $w_{\widetilde{\gamma}}$ where $\widetilde{\gamma}$ is a lift of $\gamma$ that crossing $\mathcal{P}$ and with coding $W_{\widetilde{\gamma}}=...w_{\widetilde{\gamma}}w_{\widetilde{\gamma}}...$

\end{paragraph}
\begin{paragraph}{Self-intersection of geodesics in $S$:}
Here, we explain how to compute self-intersection using the coding of a closed geodesic and the relation between self-intersection and combinatorial length.

Two lifts $\widetilde{\gamma}_1$ and $\widetilde{\gamma}_2$, with endpoints $(\widetilde{\gamma}_1^-,\widetilde{\gamma}_1^+)$ and $(\widetilde{\gamma}_2^-,\widetilde{\gamma}_2^+)$ respectively, intersect if and only if their endpoints  are cyclically ordered like $$\widetilde{\gamma}_1^-<\widetilde{\gamma}_2^-<\widetilde{\gamma}_1^+<\widetilde{\gamma}_2^+ \quad(*).$$

The self-intersection number of a closed geodesic $\gamma$ is equal to the number of time its lifts intersect in the fundamental domain $\mathcal{P}$.

\begin{figure}[htbp]
\begin{center}
\includegraphics[scale=0.35]{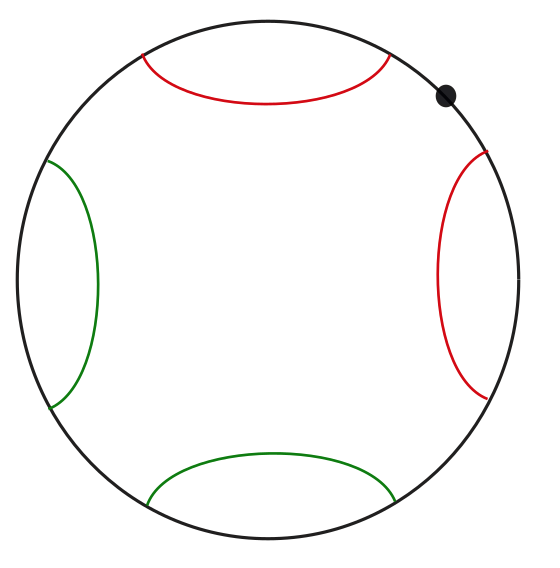}\put(-17,82){\small{$A$}}\put(-7,50){\footnotesize{$D(\bar{a})$}} \put(-55,100){\footnotesize{$D(a)$}}\put(-118,50){\footnotesize{$D(\bar{b})$}}\put(-55,-4){\footnotesize{$D(b)$}}\hspace{2cm}
\includegraphics[scale=0.35]{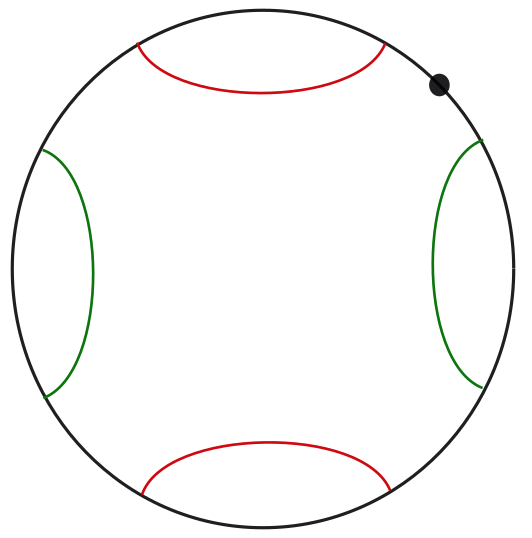}\put(-14,80){\small{$A$}}
\put(-2,50){\footnotesize{$D(b)$}} \put(-55,98){\footnotesize{$D(a)$}}\put(-115,50){\footnotesize{$D(\bar{b})$}}\put(-55,-8){\footnotesize{$D(\bar{a})$}}
\caption{\footnotesize{The order $<_A$ defined by a fixed point $A\in \partial\D$. In the hyperbolic torus, we obtain $a<_A\bar{b}<_A\bar{a}<_Ab$. For the pair of pants, we obtain $a<_A\bar{b}<_Ab<_A\bar{a}$.}}
\label{ordre}
\end{center}
\end{figure}
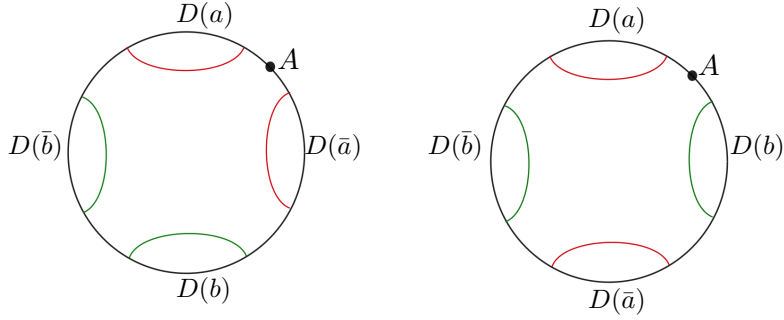

\noindent An {\bf \textit{alphabet}} is a finite ordered set of distinct symbols. For $A\in\partial \D\cap\overline{\mathcal{P}}$ fixed, we obtain an order $<_A$ on the set $\overline{\Gamma}=\{a,\bar{a},b,\bar{b}\}$ (and then an alphabet) by walking counter-clockwise on $\partial \D$ starting at $A$ and recording the order in which the domains $D(a)$, $D(\bar{a})$, $D(b)$ and $D(\bar{b})$ appear (See Figure \ref{ordre}).  
Depending on the position of $A$, we obtain four different alphabets. We denote by $\Gamma_e$ the alphabet whose first letter is $e$.
This order induces an order in $W^+$ as follow:
$$e_0e_1...<_A e'_0e'_1...\iff 
\left\{\begin{array}{ll}
  e_1<_A e'_1 \ \  \ if\ e_1\neq e'_1; or
 \\
 \\
    \exists\hspace{0.1cm} k>1, e_i=e'_i\hspace{0.1cm} \mathrm{for}\hspace{0.1cm} i<k, e_k<_A e'_k \  in\ \Gamma_{\bar{e}_{k-1}}
  \end{array}\right.$$

\noindent J. Birman and C. Series showed in \cite{Bis1}-Theorem A that the order ~$<_A$ on $W^+$ is compatible with the order on the endpoints.

It follows that using $(*)$, one can see when two lifts $\widetilde{\gamma}_i$ and  $\widetilde{\gamma}_j$ intersect or not using their codings. Since we are interested to intersection points in $\mathcal{P}$ between all the lifts, we can restrict to lifts of $\gamma$ that intersect the fundamental domain $\mathcal{P}$.

Let $\gamma$ be a closed geodesic on $S$ and $w_{\widetilde{\gamma}}=e_1...e_n$ its coding. Let $w_i:=e_ie_{i+1}...e_ne_1...e_{i-1}$ and $\widetilde{\gamma}_i$  the lift of $\gamma$ with extremities $\widetilde{\gamma}^+_i:=w_iw_i...$ and $\widetilde{\gamma}^-_i:=\bar{w}_i\bar{w}_i...$ 

We denote by $\mathcal{L}_{cyc}(\gamma)$ the set of lifts of $\gamma$ obtained as above. As a direct corollary of Lemma \ref{lem}, we have: 

 \bpro\label{lift}
Let $\gamma$ be a closed geodesic on $S$. A lift $\widetilde{\gamma}$ of $\gamma$ intersects the fundamental domain $\mathcal{P}$ if and only if $\widetilde{\gamma}\in\mathcal{L}_{cyc}(\gamma)$.
 \epro

For a closed geodesic $\gamma$ on $S$, we denote by $$I_\gamma=\{(\widetilde{\gamma}_i;\widetilde{\gamma}_j), \widetilde{\gamma}_i\cap\widetilde{\gamma}_j\neq\emptyset, \widetilde{\gamma}_i, \widetilde{\gamma}_j\in\mathcal{L}_{cyc}(\gamma), 1\leq i<j\leq L(\gamma)\}$$ the set of intersection points between lifts of $\gamma$ that intersect the fundamental domain $\mathcal{P}$. Computing the self-intersection number of $\gamma$ is equivalent to determine the number of points of $I_\gamma$ lying in the fundamental domain. But it is not easy to know whether two lifts intersect in the fundamental domain or not.
 
 \noindent In order to solve this problem, we define the following relation on $I_\gamma$: we say that two points $(\widetilde{\gamma}_i;\widetilde{\gamma}_j)$ and $(\widetilde{\gamma}_k;\widetilde{\gamma}_l)$ are equivalent and we write $(\widetilde{\gamma}_i;\widetilde{\gamma}_j)\sim (\widetilde{\gamma}_k;\widetilde{\gamma}_l)$ if there exists an isometry in $\Gamma$ which maps the pair of geodesics $(\widetilde{\gamma}_i,\widetilde{\gamma}_j)$ to the pair $(\widetilde{\gamma}_k,\widetilde{\gamma}_l)$.
 
The self-intersection number of $\gamma$ is equal to the cardinal of the quotient of $I_\gamma$ by this equivalence relation:
 $$i(\gamma;\gamma)=\#\hspace{0.1cm} I_\gamma/\sim.$$
 
 The number $i(\gamma,\gamma)$ is bounded from above by the number of pairs $(\widetilde{\gamma}_i,\widetilde{\gamma}_j)$ where $i\neq j$ and each $\widetilde{\gamma}_i$ intersects the fundamental domain. By Proposition \ref{lift} the number of $\widetilde{\gamma}_i$ intersecting the fundamental domain is equal to $L(\gamma)$. So, we have the following upper bound:
 $$i(\gamma,\gamma)\leq \frac{L(\gamma)(L(\gamma)-1)}{2}.$$
 
 The first two authors use the machinery described in this section to give  upper bounds when $S$ is a pair of pants. They proved the following:

 \bthm[Diop-Gaye, \cite{aziz2021self}]\label{D.G}
 Assume that $S$ is a pair of pants and let $\gamma$ be a primitive non-simple closed geodesic on $S$. Then,
 $$ i(\gamma;\gamma)\leq  \ds\left\{\begin{array}{ll}
 \ds\frac{L^2(\gamma)}{4} \ \ \ \ \ \ \ \ \ \ \mbox{if}  \  \mbox{$L(\gamma)$ is even}\\
 \\
 \ds\frac{L^2(\gamma)-1}{4} \ \ \ \ \ \mbox{if}  \ \mbox{ $L(\gamma)$ is odd} \\
 \end{array}\right.$$
 \ethm\vspace{0.2cm} 

Previous to that, M. Chas and A. Phillips showed, in \cite{chas2010self} a similar result in the case of a punctured torus:

 \bthm[Chas-Philips]\label{C.P}
 
 Assume that $S$ is punctured torus and let $\gamma$ be a primitive non-simple closed geodesic. Then,
 
  $$ i(\gamma;\gamma)\leq  \ds\left\{\begin{array}{ll}
  \ds\frac{\left( L(\gamma)-2\right)^2}{4} \ \ \ \ \ \ \ \ \ \ \ \ \ \ \mbox{if}  \  \mbox{$L(\gamma)$ is even}\\
  \\
  \ds\frac{\left(L(\gamma)-1\right)\left(L(\gamma)-3\right)}{4} \ \ \mbox{if}  \ \mbox{ $L(\gamma)$ is odd} \\
  \end{array}\right.$$
  Moreover, these bounds are sharp.
 \ethm
\end{paragraph}
\section {Proof of theorems \ref{thm1} and \ref{thm2}}\label{sec3}

Now, we turn to the proof of the theorems stated in the introduction. We start with the following lemmas that compute the length of combinatorial $k$-systoles in a pair of pants and a punctured torus under some condition on the choice of $k$. 

\begin{figure}[htbp]
\begin{center}
\includegraphics[scale=0.4]{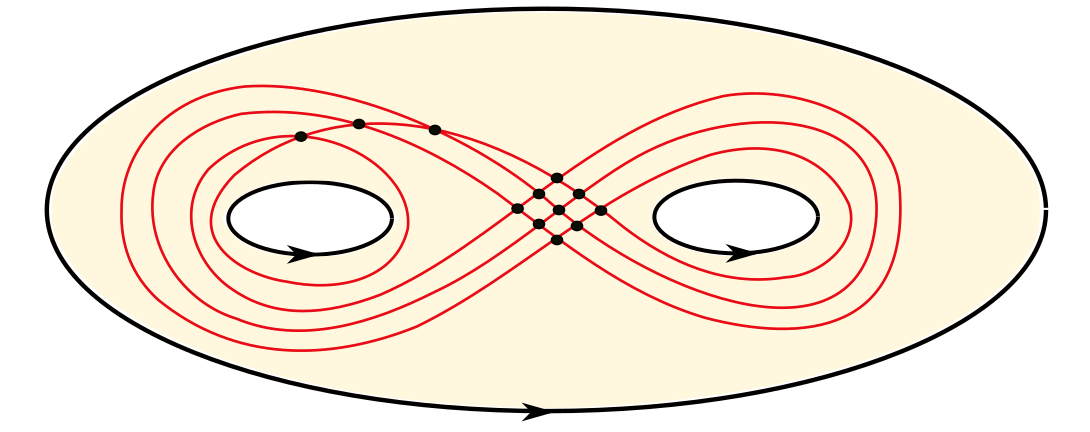}
\put(-115,-7){\small{$ab$}}
\put(-160,40){\small{$a$}}
\put(-75,40){\small{$b$}}
\caption{A geodesic on a pair of pants with coding $a(a\bar{b})^3$.}
\label{nhuit}
\end{center}
\end{figure}

\begin{lemma} [Case of a pair of pants]\label{pants}
Let $S$ be a pair of pants and $n$ an integer.

\begin{itemize}
\item If $n^2<k\leq n^2+n$, then combinatorial $k$-systoles have length equal to $2n+1$. In particular, the geodesic with coding $a(a\bar{b})^n$ is a combinatorial $k$-systole.

\item If $n^2-n< k\leq n^2$, then combinatorial $k$-systoles have length at most $2n+1$.
\end{itemize} 
\end{lemma}

\bproof

\begin{itemize}
\item Let $n$ and $k$ be two integers such that $n^2<k\leq n^2+n$. Let $\gamma_k$ be a combinatorial $k$-systole, then $$i(\gamma_k;\gamma_k)\geq k>n^2.$$ 
By Theorem \ref{D.G}, $L(\gamma_k)> 2n$. It follows that $L(\gamma_k)\geq2n+1$. The geodesic with coding $a\left(a\bar{b}\right)^n$ self-intersects $n^2+n$ times (See Figure ~\ref{nhuit} for a representation of the corresponding geodesic when $n=3$) and its combinatorial length is $2n+1$. This prove the result.

\item For $n^2-n<k\leq n^2$, the geodesic with coding $a(a\bar{b})^n$ has combinatorial length $2n+1$ and self-intersection number $n^2+n$. Therefore, combinatorial $k$-systoles have length at most $2n+1$.
\end{itemize}
\eproof

\begin{figure}[htbp]
\begin{center}
\includegraphics[scale=0.4]{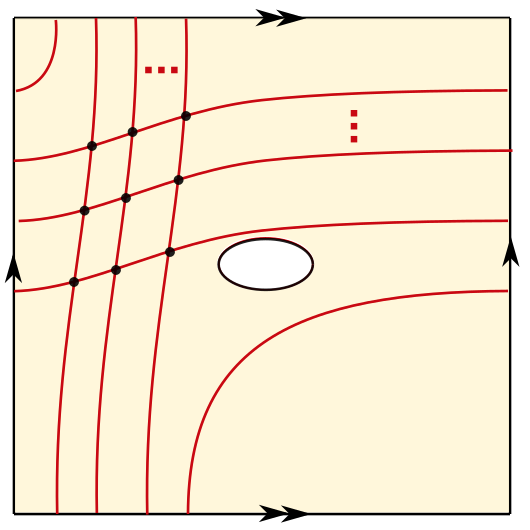}
\caption{A geodesic on the hyperbolic torus with coding $a^nb^m$}
\label{n2torus}
\end{center}
\end{figure}

\begin{lemma}[Case of a punctured torus]\label{torus}
Let $S$ be a punctured torus and $n$ an integer.

\begin{itemize}

\item  If $n^2<k\leq n^2+n$,  then combinatorial $k$-systoles have length equal to $2n+3$. In particular, the geodesic with coding $a^{n+1}b^{n+2}$ is a combinatorial $k$-systole.

\item If $n^2-n< k\leq n^2$,  then combinatorial $k$-systoles have length equal to $2n+2$. In particular, the geodesic with coding $a^{n+1}b^{n+1}$ is a combinatorial $k$-systole.
\end{itemize}
\end{lemma}
\bproof
\begin{itemize}
\item Let $n$ be an integer and $k$ an integer such that $n^2<k\leq n^2+n$. Let $\gamma_k$ be a combinatorial $k$-systole. Therefore, $$i(\gamma_k;\gamma_k)\geq k>n^2.$$ Again, by Theorem \ref{C.P}, $L(\gamma_k)> 2n+2$ which implies that  $L(\gamma_k)\geq2n+3$. The geodesic with coding $a^{n+1}b^{n+2}$ self-intersects $n^2+n$ times (see Figure \ref{n2torus} for a representative of $a^{n+1}b^{n+2}$) and its combinatorial length is $2n+3$. This prove the result.

\item Let $n$ be an integer and $k$ an integer such that $n^2-n<k\leq n^2$. 
Let $\gamma$ be a geodesic such that $L(\gamma)<2n+2$.

By Theorem \ref{C.P},
$$i(\gamma,\gamma)\leq\frac{(2n+1-1)(2n+1-3)}{4}=n^2-n.$$
Thus, $i(\gamma,\gamma)\leq n^2-n<k$ when $L(\gamma)<2n+2$. Therefore when $n^2-n<k\leq n^2-n$, combinatorial $k$-systoles have length greater than $2n+2$. The geodesic with coding $a^{n+1}b^{n+1}$ self-intersects $n^2$ times and its combinatorial length is $2n+2$. This proves the result.
\end{itemize}
\eproof

\noindent Now, we prove our main results namely Theorem \ref{thm1} and Theorem \ref{thm2}.

\bproof[Proof of Theorem 1]
We recall that $S$ is a pair of pants or a punctured torus. 
Let $k$ be an integer.

First, for $k=n^2+n$ combinatorial $k$-systoles have self-intersection equal to $k$ for a pair pants and a punctured torus. In fact, the geodesic with coding $a(a\bar{b})^n$ (respectively $a^{n+2}b^{n+1}$) has self-intersection number equal to $n^2+n$ when $S$ is a pair of pants (respectively a punctured torus). If $\gamma$ is a closed geodesic on $S$ such that $i(\gamma,\gamma)>n^2+n$, Lemma \ref{D.G} and Lemma \ref{C.P} imply that $L(\gamma)>2n+1$  for a pair of pants or $L(\gamma)>2n+3$ in the case of a punctured torus. 

So, for $k=n^2+n$ combinatorial $k$-systoles are exact that is $I^c_k=k$ and are realized by the geodesic mention above. Thus, $$\underset{k\rightarrow+\infty}{\liminf} \hspace{0.2cm}I^c_k-k=0.$$

Now, let $k$ be an integer such that $n^2<k<n^2+n$. The geodesic $a(a\bar{b})^n$ (respectively $a^{n+1}b^{n+2}$) is a combinatorial $k$-systole for a pair of pants (respectively for a punctured torus) and it self-intersects $n^2+n$ times. So, combinatorial $k$-systoles have length $2n+1$ (respectively $2n+3$) when $S$ is a pair of pants (respectively a punctured torus). Theorem \ref{D.G} and Theorem \ref{C.P} imply that the self-intersection number of combinatorial $k$-systoles is bounded from above by $n^2+n$. It follows that $I_k^c=n^2+n$. Thus : $$\underset{k\rightarrow+\infty}{\limsup} \hspace{0.2cm}I_k^c-k=+\infty.$$
This achieves the proof of Theorem \ref{thm1}.
\eproof
\bproof[Proof of Theorem 2]
For any integer $k$, there exists an integer $n$ such that: $$n^2-n<k\leq n^2 \ \ \mbox{or}\ \ n^2<k\leq n^2+n.$$ 

Lemma \ref{pants} and Lemma \ref{torus} together with Theorem \ref{D.G} and Theorem \ref{C.P} imply that: 

$$n^2-n<k\leq I^c_k\leq n^2+n \ \ \mbox{or}\ \ n^2<k\leq I^c_k\leq n^2+n.$$

These inequalities imply that:

$$\underset{k\rightarrow+\infty}{\lim}\frac{I^c_k}{k}=1.$$
\eproof

\begin{paragraph}{Acknowledgements:} The authors would like to thank the project NLAGA for their support. We are also grateful to Hugo Parlier for his comments on the early version of this article. To the anonymous referee, the authors are also grateful for all his comments and suggestions.  

\end{paragraph}

\bibliographystyle{plain}

\bibliography{biblio}
\textbf{Departement of mathematics, Universit\'e  Cheikh Anta Diop, Dakar, Senegal.}\\
\textit{email: azizdiop55@gmail.com}

\noindent \textbf{Departement of mathematics, Universit\'e  Cheikh Anta Diop, Dakar, Senegal.}\\
\textit{email: masseye.gaye@ucad.edu.sn}

\noindent \textbf{Departement of mathematics, Universit\'e Cheikh Anta Diop, Dakar, Senegal.}\\
\textit{email\string: karimka02@hotmail.fr}
\end{document}